\documentclass{amsart}
\usepackage{graphicx}
\usepackage{xcolor}

\usepackage{fullpage}
\usepackage{geometry}
\usepackage{amsmath}
\usepackage{amsthm}
\usepackage{amsfonts}
\usepackage{amssymb}
\usepackage{textcomp}
\usepackage[english]{babel}
\usepackage{amsmath}
\newtheorem{theorem}{Theorem}

\newtheorem{lemma}[theorem]{Lemma}

\newtheorem{proposition}[theorem]{Proposition}

\theoremstyle{definition}
\newtheorem{definition}[theorem]{Definition}

 \newcommand{\zbar}{\overline{z}}

\newcommand{\D}{\mathbb{D}}
\newcommand{\R}{\mathbb{R}}
\newcommand{\N}{\mathbb{N}}

\newcommand{\C}{\mathbb{C}}
\renewcommand{\Re}{\operatorname{Re}}
\renewcommand{\Im}{\operatorname{Im}}

\numberwithin{theorem}{section} \numberwithin{equation}{section}
\usepackage{enumerate}

\begin{document}

\title{An analogue of Green's functions for quasiregular maps}

\author{Mark Broderius}
\address{Department of Mathematics, Northern Illinois University, Dekalb, IL 60115, USA}
\email{mbroderius1@niu.edu}

\author{Alastair N. Fletcher}
\address{Department of Mathematics, Northern Illinois University, Dekalb, IL 60115, USA}
\email{afletcher@niu.edu}

\date{\today}

\maketitle

\begin{abstract}
Green's functions are highly useful in analyzing the dynamical behavior of polynomials in their escaping set. The aim of this paper is to construct an analogue of Green's functions for planar quasiregular mappings of degree two and constant complex dilatation. These Green's functions are dynamically natural, in that they semi-conjugate our quasiregular mappings to the real squaring map. However, they do not share the same regularity properties as Green's functions of polynomials. We use these Green's functions to investigate properties of the boundary of the escaping set and give several examples to illustrate behavior that does not occur for the dynamics of quadratic polynomials.
\end{abstract}

\section{Introduction}
\label{sec:intro}

\subsection{Background}

The iteration theory of quadratic polynomials is a very well understood area of complex dynamics. Since the first fractal images of the Mandelbrot set appeared in the 1980s, and the work of Douady and Hubbard \cite{DH84,DH85} initiated a great deal of interest in this subject, there has been a consistent surge of interest. These fascinating images that arise from simply defined functions yields a surprisingly intricate theory that is still not completely settled. We refer to the texts of Beardon \cite{Bea91}, Carleson and Gamelin \cite{CG93}, and Milnor \cite{Mil06} for introductions to complex dynamics.

The introduction of quasiconformal and quasiregular mappings into the theory of complex dynamics was another major reason for the impetus of the 1980s. Douady and Hubbard's work on polynomial-like mappings \cite{DH85-2} and Sullivan's proof of the No Wandering Domains Theorem for rational maps \cite{Sul85} illustrated the utility of this approach. For more applications of quasiregular mappings in complex dynamics, we refer to the text of Branner and Fagella \cite{BF14}.

More recently, the study of the iteration theory of quasiregular mappings themselves has become an object of interest. This approach naturally extends complex dynamics into higher real dimensions. However, due to tools available only in two dimensions such as the Measurable Riemann Mapping Theorem and the Stoilow Decomposition Theorem, more can be said in the two dimensional case. For an introduction to quasiregular dynamics, we refer to Bergweiler's survey \cite{Ber10} as a starting point.

The subject of our study is the class of quasiregular maps of the form
\begin{equation}
\label{eq:hktc}
H_{K,\theta,c}(z) = \left [ \left ( \frac{ K+1}{2} \right )z + e^{2i\theta} \left ( \frac{K-1}{2} \right ) \zbar \right ]^2 +c ,
\end{equation}
where $K>1$, $\theta \in (-\pi/2,\pi/2]$ and $c\in \C$. The dynamics of these mappings were first studied by the second named author and Goodman in \cite{FG10}, and subsequently by the second named author and Fryer in \cite{FF12,FF16}. We will review the important points from these papers in the preliminary section, but here let us point out that these are the simplest non-injective quasiregular mappings as they have constant complex dilatation
\[ \mu_{H_{K,\theta,c}} \equiv e^{2i\theta} \left ( \frac{K-1}{K+1} \right ) \in \D .\]
It is evident that we can write
\[ H_{K,\theta,c} = P_c \circ h_{K,\theta},\]
where $P_c(z) = z^2+c$ and $h_{K,\theta}$ is the affine map that stretches by a factor $K>1$ in the direction $e^{i\theta}$. This Stoilow decomposition of $H_{K,\theta,c}$ is also a characterization of degree two quasiregular maps in $\C$ with constant complex dilatation: every such map is conformally conjugate to
one of the form \eqref{eq:hktc}. Observe that if $K=1$, then we return to the case of quadratic polynomials. One of the goals of this paper is to illustrate new phenomena that can occur in this setting, when compared to complex dynamics. Figure \ref{fig:5} gives an example where the function has a saddle fixed point, a situation that cannot occur for holomorphic functions.

\begin{figure}[h]
\begin{center}
\includegraphics[width=2.5in]{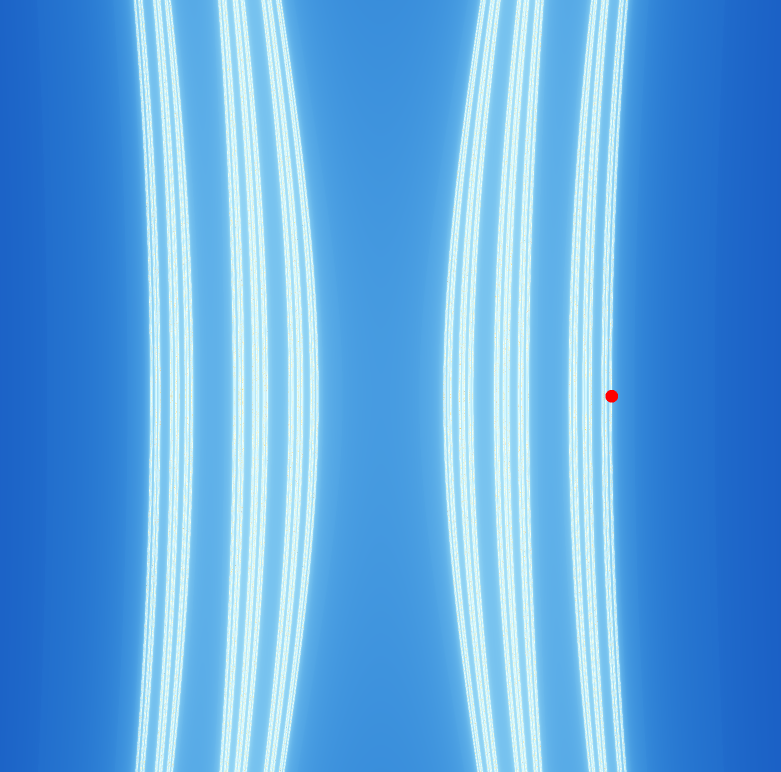}
\end{center}
\caption{The dynamical plane for $H_{5,0,-0.1}$ with the saddle fixed point marked.}
\label{fig:5}
\end{figure}

\subsection{Quadratic polynomials}

We briefly review some of the well-known material from the dynamics of quadratic polynomials to set the stage for what is to come. We recall that every quadratic polynomial is linearly conjugate to one of the form $P_c(z) = z^2+c$.

Given a non-injective polynomial $P$, $\C$ can be decomposed into the escaping set 
\[ I(P) = \{ z\in \C : P^n(z) \to \infty \} \]
and the bounded orbit set
\[ BO(P) = \{ z\in \C : \text{ there exists } M>0 \text{ such that } |P^n(z)| \leq M \text{ for all } n\in \N \}.\]
The bounded orbit set is often called the filled Julia set in the literature, and denoted by $K(P)$, but we will reserve the use of $K$ for the maximal dilatation of quasiregular mappings. 

The chaotic set is called the Julia set, and denoted by $J(P)$, whereas the stable set is called the Fatou set, and denoted by $F(P)$. Slightly more formally, the Fatou set is where the family of iterates locally forms a normal family, and the Julia set is the complement of the Fatou set. The importance of the escaping set is that the definition of the Julia set is a difficult one to check, but the fact that
\[ J(P) = \partial I(P) \]
gives a highly useful way to visualize the Julia set. Critical fixed points of polynomials are called superattracting fixed points and are necessarily in the Fatou set. In particular, if we extend $P$ to the Riemann sphere $\C_{\infty}$, then we see that $\infty$ is a superattracting fixed point of $P$.

For quadratic polynomials, we have a dichotomy in the dynamical behavior. If $0\in BO(P)$, then both $BO(P)$ and $J(P)$ are connected. On the other hand, if $0\in I(P)$ then both $BO(P)$ and $J(P)$ are a Cantor set, that is, a totally disconnected, compact, perfect set. These distinctive behaviors give rise to the definition of the Mandelbrot set in parameter space:
\[ \mathcal{M} = \{c\in \C : 0 \in BO(P_c) \}.\]

B\"ottcher's Theorem allows us to conformally conjugate near a superattracting fixed point of local index $d$ to the power map $z\mapsto z^d$. For quadratic polynomials, as $\infty$ is a superattracting fixed point, we may conformally conjugate to $z\mapsto z^2$ in a neighborhood of $\infty$. More precisely, there exists a conformal map $\varphi_c$ and a neighborhood $W_c$ of $\infty$ such that
\begin{equation}
\label{eq:bottcher} 
\varphi_c ( P_c(z)) = \left [ \varphi_c(z) \right ]^2 
\end{equation}
for $z\in W_c$. It is desirable to extend \eqref{eq:bottcher} to the largest possible domain. The only obstruction to extending $W_c$ to larger domains is if the orbit of the critical point at $0$ escapes. Therefore, if $0\in \mathcal{M}$, then \eqref{eq:bottcher} holds on all of $I(P_c)$, whereas if $0\notin \mathcal{M}$, then we cannot do so.

A resolution to this problem is to use harmonic Green's functions. The real valued function
\[ G_c(z) = \log | \varphi_c(z)| \]
satisfies the functional equation
\begin{equation}
\label{eq:greenpc} 
G_c(P_c(z)) = 2 G_c(z) 
\end{equation}
initially in $W_c$. This time the functional equation allows us to extend the domain of definition of $G_c$ to all of $I(P_c)$, and it turns out that $G_c$ is precisely the Green's function for the exterior domain of $BO(P_c)$ with a pole at $\infty$. In particular, $G_c$ is harmonic and $G_c(z) \to 0$ as $\operatorname{dist}(z,J(P_c)) \to 0$.

The equipotentials $G_c(z) = t$ give a particularly striking picture of the dynamics of $P_c$. If $c\in \mathcal{M}$, then these equipotentials give a foliation of $I(P_c)$ through simple closed curves. On the other hand, if $c\notin \mathcal{M}$, then while the equipotentials for $t$ large enough are simple closed curves, there is a critical value given by $G_c(0) = t_0$ for which the level curve forms a figure-eight shape. For $0<t<t_0$, the level curves are then disconnected.

\subsection{Statement of results}

For mappings of the form \eqref{eq:hktc}, the plane can be again decomposed into the escaping set and the bounded orbit set.

In \cite{FF12}, an analogous result to B\"ottcher's Theorem is established with gives the existence of a quasiconformal map $\varphi_{K,\theta,c}$ which conjugates $H_{K,\theta,c}$ to $H_{K,\theta,0}$ on a neighborhood $W_{K,\theta,c}$ of $\infty$ contained in the escaping set, that is,
\[ \varphi_{K,\theta,c} ( H_{K,\theta,c} (z)) = H_{K,\theta,0} ( \varphi_{K,\theta,c} (z))\]
for $z\in W_{K,\theta,c}$. The question then arises of whether this conjugation can be extended. Once again, the Mandelbrot set in parameter space plays an important role. For $K>1$ and $\theta \in (-\pi/2 , \pi/2]$, we define the Mandelbrot set
\[ \mathcal{M}_{K,\theta} = \{ c\in \C : 0 \in BO(H_{K,\theta,c}) \} .\]
We note that this time there is not an equivalent formulation in terms of the connectedness of the Julia set, as $J(H)$ and $\partial I(H)$ may differ for quasiregular maps, see for example \cite[Example 7.3]{BN14}.

Returning to the B\"ottcher coordinate, it was shown in \cite[Theorem 2.4]{FF12} that if $c\in \mathcal{M}_{K,\theta}$, then $\varphi_{K,\theta,c}$ may be extended to a locally quasiconformal map on all of $I(H_{K,\theta,c})$, whereas if $c\notin \mathcal{M}_{K,\theta}$, then it cannot. Our first main result shows that the Green's function idea also works in this setting.

\begin{theorem}
\label{thm:green} 
Let $K>1$, $\theta \in (-\pi/2 , \pi/2]$ and $c\in \C$.
There exists a non-negative, continuous function $G_{K,\theta,c}:\mathbb{C}\to\mathbb{R}$ that is identically zero on $BO(H_{K,\theta,c})$, non-zero on $I(H_{K,\theta,c})$ and such that 
\[ G_{K,\theta,c}(H_{K,\theta,c}(z))=2G_{K,\theta,c}(z) \]
for all $z\in \C$.
\end{theorem}

We do not expect $G_{K,\theta,c}$ to be harmonic, although we leave the question of the regularity of this function to future work.
However, $G_{K,\theta,c}$ still has equipotentials that are useful dynamically. For $t>0$, let us define 
\[ E(t) = E_{K,\theta,c}(t) = \{z\in \C : G_{K,\theta,c}(z) = t \},\]
and
\[ U(t) = U_{K,\theta,c}(t) = \{ z\in \C : G_{K,\theta,c}(z) > t \}.\]
Using these notions, we will give an alternative proof and mild refinement of \cite[Theorem 5.3 and 5.4]{FG10}.

\begin{theorem}
\label{thm:conn}
Let $K>1$, $\theta \in (-\pi/2,\pi/2]$ and $c\in \C$.
If $c\in \mathcal{M}_{K,\theta}$, then $\partial I(H_{K,\theta,c})$ is connected. If $c\notin \mathcal{M}_{K,\theta}$, then $\partial I(H_{K,\theta,c})$ has uncountably many components.
\end{theorem}

Recalling the dichotomy that $J(P_c)$ is either connected or a Cantor set, we now turn to the question of what happens when $\partial I(H_{K,\theta,c})$ has uncountably many components. First, we give a condition that ensures that $\partial I(H_{K,\theta,c})$ is not a Cantor set. Recall that a periodic point of period $n$ for $f$ is a solution of $f^n(z) = z$. Note that a fixed point is considered a periodic point of all periods.

\begin{theorem}
\label{thm:perpts}
Let $K>1$, $\theta \in (-\pi/2,\pi/2]$ and $c\notin \mathcal{M}_{K,\theta}$. Suppose there exists $n\in \N$ such that $H_{K,\theta,c}$ has at least $2^n+1$ periodic points of period $n$. Then $\partial I(H_{K,\theta,c})$ is not a Cantor set.
\end{theorem}

Finally, we aim to illustrate that the situation in Theorem \ref{thm:perpts} can indeed occur and exhibit some of the other features that can occur for the dynamics of these mappings that contrast with the dynamics of quadratic polynomials.

\begin{theorem}
\label{thm:examples}
We may choose parameters $K>1$, $\theta \in (-\pi/2,\pi/2]$ and $c\in \C$ such that each of the following cases may occur:
\begin{enumerate}[(a)]
\item $H_{K,\theta,c}$ has either two, three or four fixed points in $\C$, and four is the maximum possible.
\item $H_{K,\theta,c}$ has an attracting fixed, yet $c\notin \mathcal{M}_{K,\theta}$, that is, there is an attracting fixed point in one of uncountably many components of the bounded orbit set.
\item Given $K,\theta$, there exist parameters $c\in \C$ such that $H_{K,\theta,c}$ has a saddle fixed point. Moreover, there exist parameters $K,\theta,c$ such that $H_{K,\theta,c}$ has a saddle fixed point $z_0$, and the intersection of the component of $BO(H_{K,\theta,c})$ containing $z_0$ with an open neighbourhood of $z_0$ is a smooth curve.
\end{enumerate}
\end{theorem}

\begin{figure}[h]
\begin{center}
\includegraphics[width=5in]{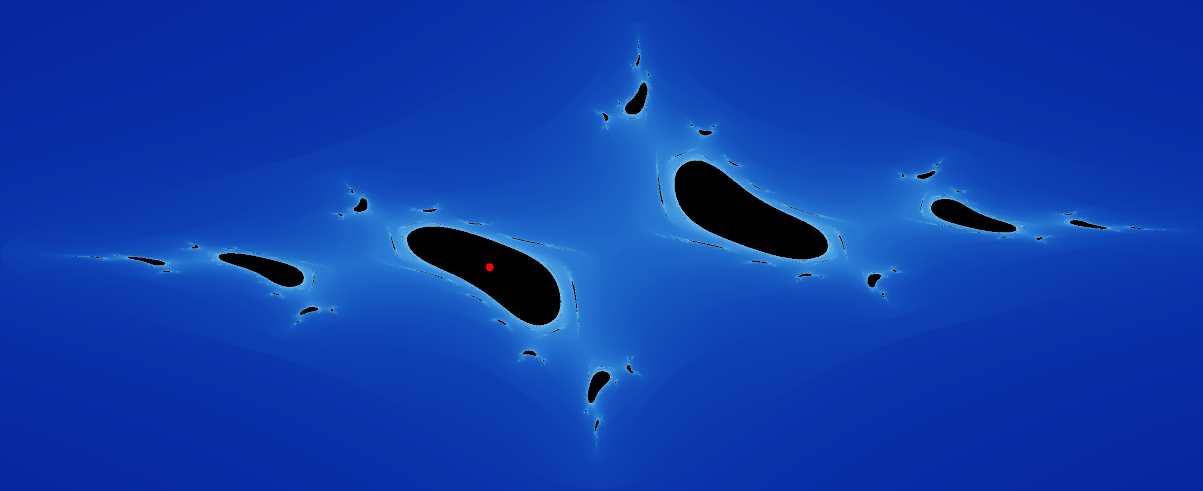}
\end{center}
\caption{The dynamical plane of $H_{1/2,0,-3/2-i/2}$ with the attracting fixed point $z_0$ marked.}
\label{fig:3}
\end{figure}

Figure \ref{fig:3} illustrates an example of the second case in Theorem \ref{thm:examples}, and Figure \ref{fig:5} illustrates an example of the third case.

The paper is organized as follows. In Section \ref{sec:prelims}, we recall preliminary material on quasiregular mappings in the plane, and known material on the dynamics of the mappings $H_{K,\theta,c}$. In Section \ref{sec:greens} we prove Theorem \ref{thm:green} by constructing our analogue of a Green's function. In Section \ref{sec:equi} we use equipotentials for our Green's functions to establish Theorem \ref{thm:conn}. In Section \ref{sec:fixed} we prove Theorem \ref{thm:perpts}. In Section \ref{sec:classify}, we give a classification of the type of fixed point of $H_{K,0,c}$ based on its location and finally in Section \ref{sec:examples} we use the classification to help establish the examples in Theorem \ref{thm:examples}.

\section{Preliminaries}
\label{sec:prelims}

\subsection{Quasiregular mappings}

A quasiconformal mapping $f:\C \to\C$ is a homeomorphism such that $f$ is in the Sobolev space $W^1_{2,loc}(\C)$ and there exists $k\in [0,1)$ such that the complex dilatation $\mu_f = f_{\zbar}/f_z$ satisfies
\[ |\mu_f(z)| \leq k \]
almost everywhere in $\C$. The dilatation at $z\in \C$ is
\[ K_f(z) := \frac{1+|\mu_f(z)|}{1-|\mu_f(z)|}.\]
A mapping is called $K$-quasiconformal if $K_f(z) \leq K$ almost everywhere. The smallest such constant is called the maximal dilatation and denoted by $K(f)$. If we drop the assumption on injectivity, then $f$ is called a quasiregular mapping. 
We refer to the books of Rickman \cite{Ric93} and Iwaniec and Martin \cite{IM01} for much more on the development of the theory of quasiregular mappings.

In the plane, we have the following two crucial results. First, there is a surprising correspondence between quasiconformal mappings and measurable functions, see for example \cite[p.8]{IM01}.

\begin{theorem}[Measurable Riemann Mapping Theorem]
Suppose that $\mu \in L^{\infty}(\C)$ with $||\mu||_{\infty} \leq k <1$. Then there exists a quasiconformal map $f:\C \to \C$ with complex dilatation equal to $\mu$ almost everywhere. Moreover, $f$ is unique if it fixes $0,1$ and $\infty$.
\end{theorem}

Moreover, every quasiregular mapping has an important decomposition, see for example \cite[p.254]{IM01}.

\begin{theorem}[Stoilow Decomposition Theorem]
\label{thm:stoilow}
Let $f:\C\to \C$ be a quasiregular mapping. Then there exists a holomorphic function $g$ and a quasiconformal mapping $h$ such that $f=g\circ h$.
\end{theorem}

In fact, in the plane, some sources use this decomposition as the definition of a quasiregular mapping. However, this approach does not generalize to higher dimensions, so even though this paper is two dimensional, we will keep Stoilow Decomposition as a theorem.

\subsection{The mappings $H_{K,\theta,c}$}

Here we review the properties of the mappings we will focus on in this paper. Let $K>1$ and $\theta \in (-\pi/2 , \pi/2]$. Then we define $h_{K,\theta}$ to be the stretch by factor $K$ in the direction of $e^{i\theta}$. Evidently, if $\theta =0$, then 
\[ h_{K,0}(x+iy) = Kx+iy.\]
Interpreting $h_{K,\theta}$ in terms of $h_{K,0}$, if $\rho_{\theta}$ denotes the rotation counter-clockwise by angle $\theta$, then we see that 
\[ h_{K,\theta} = \rho_{\theta} \circ h_{K,0} \circ \rho_{-\theta}.\]
From this we can obtain the following explicit formulas for $h_{K,\theta}$:
\begin{align*}
h_{K,\theta}(z) &= \left ( \frac{K+1}{2} \right ) z + e^{2i\theta} \left (\frac{K-1}{2} \right ) \zbar \\
&= \left [ x(K\cos^2(\theta)+\sin^2(\theta))+y(K-1)\sin(\theta)\cos(\theta) \right ] \\
& \hskip0.5in + i\left [ x(K-1)\cos(\theta)\sin(\theta)+y(K\sin^2(\theta)+\cos^2(\theta))\right ].
\end{align*}
It is clear that
\[ \mu_{h_{K,\theta}} \equiv e^{2i\theta} \left ( \frac{K-1}{K+1} \right ),\]
and it follows from the Measurable Riemann Mapping Theorem and \cite[Proposition 3.1]{FG10} that any quasiconformal mapping of $\C$ with constant complex dilatation given as above arises by conjugating $h_{K,\theta}$ by a complex linear map.

If we view quasiconformal mappings with constant complex dilatation as the simplest in their class, then we can view quadratic polynomials as the simplest non-injective holomorphic functions. In light of the Stoilow Decomposition Theorem, the simplest quasiregular mappings arise as compositions of quadratic polynomials and mappings of the form $h_{K,\theta}$. As defined in the introduction, for $K>1$, $\theta \in (-\pi/2,\pi/2]$ and $c\in \C$, we have $H_{K,\theta,c} = P_c \circ h_{K,\theta}$. For later use, we note that if $r>0$, then
\begin{equation}
\label{eq:hlin} 
h_{K,\theta}(rz) = rh_{K,\theta}(z) \text{ and } H_{K,\theta,0}(rz) = r^2H_{K,\theta,0}(z).
\end{equation}

It is worth remarking on the domains of $K$ and $\theta$. Evidently a stretch in the direction $e^{i\theta}$ is the same as a stretch in the direction $e^{i(\theta + \pi)}$, and so we only need consider $\theta \in (-\pi/2,\pi/2]$. A stretch by factor $K>1$ in the direction $e^{i\theta}$ is conjugate to a contraction by factor $1/K$ in the direction $e^{i(\theta + \pi/2)}$. We could thus restrict the domain of $\theta$ further and allow any positive value for $K$, but we will usually use the convention that $K>1$ instead. There may be occasion to conveniently allow $K>0$, but we will alert the reader when this is the case.
We recall from \cite[Proposition 3.1]{FG10} that every degree two quasiregular mapping of the plane with constant complex dilatation is linearly conjugate to $H_{K,\theta,c}$ for some choice of parameters.

\subsection{Dynamics of $H_{K,\theta,c}$}

As $H_{K,\theta,c}$ is quasiregular, we may consider the behavior of its iterates. In the plane, every uniformly quasiregular mapping, that is, one for which there is a uniform bound on the maximal dilatation of the iterates, is known to be a quasiconformal conjugate of a holomorphic map. This means that the dynamics of such mappings yields essentially nothing new when compared to the features of complex dynamics. Importantly, from the point of view of independent interest, the mappings $H_{K,\theta,c}$ are not uniformly quasiregular \cite[Theorem 1.12]{FF16}.

The escaping set for $H_{K,\theta,c}$ is a non-empty, open neighborhood of $\infty$ by \cite[Theorem 4.3]{FG10}. It follows that the bounded orbit set is the complement of the escaping set. All of $I(H_{K,\theta,c}), BO(H_{K,\theta,c})$ and $\partial I(H_{K,\theta,c})$ are completely invariant under $H_{K,\theta,c}$.

When $c=0$, we can guarantee a neighborhood of $0$ is in $BO(H_{K,\theta, 0})$.

\begin{lemma}
\label{lem:bo}
Let $K>1$ and $\theta \in (-\pi/2,\pi/2]$. Then the ball $\{ z : |z| < (2K^2)^{-1} \}$ is contained in $BO(H_{K,\theta,0})$.
\end{lemma}

\begin{proof}
If $|z| < (2K^2)^{-1}$, then we have 
\begin{align*}
|H_{K,\theta,0}(z)| &=  \left | \left ( \frac{ K+1}{2} \right )z + e^{2i\theta} \left ( \frac{K-1}{2} \right ) \zbar \right |^2 \\
&\leq K^2|z|^2 \\
&< \frac{K^2|z|}{2K^2}\\
&=\frac{|z|}{2}.
\end{align*}
The conclusion follows.
\end{proof}

We also have control of the growth of $H_{K,\theta,c}$ near infinity.

\begin{lemma}
\label{lem:esc}
Let $K>1$, $\theta \in (-\pi/2,\pi/2]$ and $c\in \C$.
If $|z| \geq \max \{ |c| ,2 \}$, then 
\[ \frac{1}{2} \leq \frac{|H_{K,\theta,c}(z)|}{|z|^2} \leq (K^2+1).\]
\end{lemma}

\begin{proof}
If $|z| \geq \max \{|c| , 1 \}$, then
\begin{align*}
|H_{K,\theta,c} (z)| &= \left | h_{K,\theta}(z)^2 + c \right | \\
&\leq |h_{K,\theta}(z)|^2 + |c| \\
&\leq K^2|z|^2 + |z| \\
&\leq (K^2+1)|z|^2,
\end{align*}
which gives the upper bound. For the lower bound, if $|z| \geq \max \{ |c| , 2\}$, then
\begin{align*}
|H_{K,\theta,c}(z)| &\geq |h_{K,\theta}(z)|^2 - |c| \\
&\geq |z|^2 - |z|\\
&= |z|^2(1 - |z|^{-1})\\
&\geq  \frac{|z|^2}{2},
\end{align*}
as required.
\end{proof}

As this lemma also illustrates, the escaping set of $H_{K,\theta,c}$ is non-empty, and we can therefore ask for a conjugation to a simpler mapping in a neighborhood of infinity, analogous to B\"ottcher's Theorem. The following result yields this.

\begin{theorem}[Theorem 2.1, \cite{FF12}]
\label{thm:bottcher}
Let $K>1$, $\theta \in (-\pi/2,\pi/2]$ and $c\in \C$. Then there exist a neighborhood of infinity $W = W_{K,\theta,c}$ and a quasiconformal map $\varphi = \varphi_{K,\theta,c}$ defined in $W$ such that
\[ \varphi \circ H_{K,\theta,c} = H_{K,\theta,0} \circ \varphi \]
holds in $W$.
\end{theorem}

\subsection{Riemann-Hurwitz formula}

Finally, we will need the following version of the Riemann-Hurwitz formula. Due to the Stoilow Decomposition Theorem, there are no technical issues when applying the Riemann-Hurwitz formula with quasiregular mappings instead of holomorphic functions, but we refer to \cite[Corollary 5.2]{FG10} for more details.

\begin{theorem}
\label{thm:rh}
Let $D_1$ and $D_2$ be domains in $\C_{\infty}$ whose boundaries consist of a finite number of simple closed curves. Let $f$ be a proper quasiregular map of degree $d$ from $D_{1}$ onto $D_{2}$ with $L$ branch points including multiplicity. Then every $z\in D_{2}$ has the same number $d$ of pre-images including multiplicity and 
\begin{equation*}
2-d_{1}=d(2-d_{2})-L,
\end{equation*}
where $d_{j}$ is the number of boundary components of $D_{j}$.
\end{theorem}

\section{Constructing the Green's function}
\label{sec:greens}

In this section, we will construct our analogue of Green's function and prove Theorem \ref{thm:green}. First, let us denote by $R_{\phi}$ the ray $\{ te^{i\phi} : t \geq  0 \}$. The bounded orbit set $BO(H_{K,\theta,0})$ is starlike about $z=0$, as the next lemma shows. We remark that this result has appeared as \cite[Corollary 1.10]{FF16}, although there is some opacity to that proof which the following proof makes transparent.

\begin{lemma}
\label{lem:star}
Let $K>1$ and $\theta \in (-\pi/2,\pi/2]$. For any $\phi \in [0,2\pi)$,
the set $\partial I(H_{K,\theta,0})\cap R_{\phi}$ contains exactly one element.
\end{lemma}

\begin{proof}
	  The set $I(H_{K,\theta,0})$ contains a neighborhood of infinity, so $I(H_{K,\theta,0})\cap R_{\phi}$ is non-empty. By Lemma $\ref{lem:bo}$, there exists $r>0$ such that $\{z:|z|<r\}$ is in the interior of $BO(H_{K,\theta,0})$. Thus, $BO(H_{K,\theta,0})\cap R_{\phi}$ is non-empty. As such, $\partial H_{K,\theta,0}\cap R_{\phi}$ contains at least one element. Moreover, $\partial I(H_{K,\theta,0})$ contains no elements that are less than $r$ in absolute value.\\
	 
	 Suppose towards a contradiction that $\partial I(H_{K,\theta,0})\cap R_{\phi}$ contains at least two elements. Then there are elements $z_1$ and $z_2$ in $\partial I(H_{K,\theta,0})\cap R_{\phi}$ such that $z_1=tz_2$, where $t>1$. By \eqref{eq:hlin}, Lemma \ref{lem:bo}, and the complete invariance of $\partial I(H_{K,\theta,c})$, for $n\in \N$ we have
	  \begin{align*}|H^{n}_{K,\theta,0}(z_1)|&=|H^n_{K,\theta,0}(tz_2)|\\
	 	& =|t^{2^n}H^{n}_{K,\theta,0}(z_2)|\\
	 	&=t^{2^n}|H^{n}_{K,\theta,0}(z_2)|\\
	 	&\geq \frac{ t^{2^n}}{2K^2}.
	\end{align*}
Therefore, $z_{1}\in I(H_{K,\theta,0})$, which is a contradiction.
\end{proof}

In light of Lemma \ref{lem:star}, we make the following definition.

\begin{definition}
\label{def:bkt}
Let $K>1$ and $\theta \in (-\pi/2 , \pi/2]$. Denote by $b_{K,\theta}(\phi)\in \C$ the unique element of $\partial I(H_{K,\theta,0})\cap R_{\phi}$. 
\end{definition}

Our aim is to use $b_{K,\theta}$ to model the dynamics of $H_{K,\theta,0}$ on a squaring map.

\begin{definition}
\label{def:tau0}
Let $K>1$ and $\theta \in (-\pi/2,\pi/2]$.
Define $\tau_{K,\theta,0}:\C \to \R^+$ for $z\neq 0$ by $\tau_{K,\theta,0}(z)=\frac{z}{b_{K,\theta}(\arg(z))}$, and set $\tau_{K,\theta,0}(0)=0$.
\end{definition}

Observe that $\arg ( b_{K,\theta}(\arg(z)) ) = \arg (z)$.
The key point here is that $\tau_{K,\theta,0}^{-1} (1) = \partial I(H_{K,\theta , 0})$.

\begin{lemma}
\label{lem:tau scale}
For any $r>0$, $\tau_{K,\theta,0}(rz)=r\tau_{K,\theta,0}(z)$. 
\end{lemma}

\begin{proof}
The claim is clear if $r=0$.
Otherwise, 
\begin{align*}
\tau_{K,\theta,0}(rz) &= \frac{rz}{b_{K,\theta}(\arg(z))}\\
&=r\tau_{K,\theta,0}(z),
\end{align*}
as required.
\end{proof}

The following lemma is the whole point of introducing $\tau_{K,\theta ,0}$: it semi-conjugates between $H_{K,\theta,0}$ in $\C$ and the squaring map on $\R^+$.

\begin{lemma}
\label{lem:tau squared}
Let $K>1$ and $\theta \in (-\pi/2,\pi/2]$.
For any $z\in \mathbb{C}$, 
\begin{equation*}
\tau_{K,\theta,0}(H_{K,\theta,0}(z))=[\tau_{K,\theta,0}(z)]^2.
\end{equation*}
\end{lemma}

\begin{proof}
If $z=0$, then the lemma is clear. Otherwise, suppose that $z=re^{i\phi}$. Then by \eqref{eq:hlin} and Lemma \ref{lem:tau scale}, we have
\begin{align*}
\tau_{K,\theta,0}(H_{K,\theta,0}(re^{i\phi}))&=\tau_{K,\theta,0}\left ( H_{K,\theta,0}\left (\frac{r|b_{K,\theta}(\phi)|e^{i\phi}}{|b_{K,\theta}(\phi)|}\right )\right )\\
&=\frac{r^2}{|b_{K,\theta}(\phi)|^2}\tau_{K,\theta,0}(H_{K,\theta,0}(|b_{K,\theta}(\phi)|e^{i\phi}))\\
&=\frac{r^2}{|b_{K,\theta}(\phi)|^2}\tau_{K,\theta,0}(H_{K,\theta,0}(b_{K,\theta}(\phi)))
\end{align*}
As $b_{K,\theta}(\phi) \in \partial I(H_{K,\theta,0})$ and as $\partial I(H_{K,\theta,0})$ is completely invariant, we see that $\tau_{K,\theta,0}(H_{K,\theta,0}(b_{K,\theta}(\phi)))  = \tau_{K,\theta,0}(b_{K,\theta}(\phi) )= 1$.
Therefore,
\begin{align*}
\tau_{K,\theta,0}(H_{K,\theta,0}(re^{i\phi}))&=\frac{r^2}{|b_{K,\theta}(\phi)|^2}\\
&=\left(\frac{r}{|b_{K,\theta}(\phi)|}\tau_{K,\theta,0}(b_{K,\theta}(\phi))\right)^2\\
&=\left[\tau_{K,\theta,0}\left(\frac{r}{|b_{K,\theta}(\phi)|}b_{K,\theta}(\phi)\right)\right]^2\\
&=[\tau_{K,\theta,0}(re^{i\theta})]^2.
\end{align*}
\end{proof}

We next extend the definition of $\tau_{K,\theta,0}$ to $\tau_{K,\theta,c}$ for general $c\in \C$ by using the B\"ottcher coordinate, with the caveat that $\tau_{K,\theta,c}$ is not (immediately) globally defined. Recall the quasiconformal map $\varphi_{K,\theta,c}$ defined in the neighborhood $W_{K,\theta,c}$ of $\infty$ from Theroem \ref{thm:bottcher}.

\begin{definition}
\label{def:tauc}
Let $K>1$, $\theta \in (-\pi/2,\pi/2]$ and $c\in \C$. For $z\in W_{K,\theta,c}$, we define $\tau_{K,\theta,c} : W_{K,\theta,c} \to \R^+$ by
\[ \tau_{K,\theta,c}(z) = \tau_{K,\theta,0} ( \varphi_{K,\theta,c} (z) ).\]
\end{definition}

Importantly, this generalized version of $\tau_{K,\theta,0}$ still semi-conjugates between $H_{K,\theta,c}$ in a neighborhood of $\infty$ and the squaring map.

\begin{lemma}
\label{lem:tau general squared}
Let $K>1$, $\theta \in (-\pi/2,\pi/2]$ and $c\in \C$.
For any $z\in W_{K,\theta,c}$, 
\begin{equation*}
\tau_{K,\theta,c}(H_{K,\theta,c}(z))=[\tau_{K,\theta,c}(z)]^2.
\end{equation*}
\end{lemma}

\begin{proof}
By Lemma \ref{lem:tau squared} and Theorem \ref{thm:bottcher}, for $z\in W_{K,\theta,c}$, we have
\begin{align*}
\tau_{K,\theta,c}(H_{K,\theta,c}(z))&=\tau_{K,\theta,0}\circ (\varphi_{K,\theta,c}\circ H_{K,\theta,c})(z)\\
&=\tau_{K,\theta,0}\circ (H_{K,\theta,0} \circ \varphi_{K,\theta,c})(z)\\
&=(\tau_{K,\theta,0}\circ H_{K,\theta,0}) \circ \varphi_{K,\theta,c}(z)\\
&=[\tau_{K,\theta, 0}(\varphi_{K,\theta,c}(z))]^2\\
&=[\tau_{K,\theta,c}(z)]^2.
\end{align*}
\end{proof}

Recall from \eqref{eq:greenpc} that the Green's function $G_c$ semi-conjugates between $P_c$ and multiplication by $2$. To construct our analogue of the Green's function, we may now just modify $\tau_{K,\theta,c}$ as follows.

\begin{definition}
\label{def:gktc}
Let $K>1$, $\theta \in (-\pi/2,\pi/2]$ and $c\in \C$. We define $G_{K,\theta,c} : W_{K,\theta,c} \to \R^+$ by
\[ G_{K,\theta,c}(z) = \log (\tau_{K,\theta,c}(z)).\]
\end{definition}

It is then clear from Lemma \ref{lem:tau general squared} that for $z\in W_{K,\theta,c}$, we have
\begin{equation}
\label{eq:gktcsemiconj}
G_{K,\theta,c}(H_{K,\theta,c}(z)) = 2G_{K,\theta,c}(z).
\end{equation}

We are now a long way towards proving Theorem \ref{thm:green}. The remaining task is to show that we can extend the domain of definition of $G_{K,\theta,c}$ to all of $\C$. Recall that the backward orbit of a set $X$ under a map $f$ is 
\[ O_f^-(X) = \bigcup_{n=0}^{\infty} f^{-n}(X).\]
If the context is clear, we will just write $O^-(X)$.

\begin{lemma}
\label{lem:Big Union}
Let $K>1$, $\theta \in (-\pi/2,\pi/2]$ and $c\in \C$.
If $U$ is a neighborhood of infinity contained in $I(H_{K,\theta,c})$, then 
\[ I(H_{K,\theta,c})=O^-(U).\]
\end{lemma}

\begin{proof}
By the complete invariance of the escaping set, it is clear that $O_{H_{K,\theta,c}}^-(U) \cap BO(H_{K,\theta,c}) = \emptyset$ and so 
\[ O^-(U) \subset I(H_{K,\theta,c}).\]
On the other hand, if $z\in I(H_{K,\theta,c})$, then for a large enough iterate, $f^n (z) \in U$ and so
\[ I(H_{K,\theta,c}) \subset O^-(U),\]
which completes the proof.
\end{proof}

We can now complete the proof of Theorem \ref{thm:green}

\begin{proof}[Proof of Theorem \ref{thm:green}]
From Lemma \ref{lem:tau general squared} and \eqref{eq:gktcsemiconj}, we have the existence of $G_{K,\theta,c}$ in $W_{K,\theta,c}$ and the required functional equation is satisfied there. The idea is to pullback the domain of definition of $G_{K,\theta,c}$ via the functional equation. That is, if $z\in H_{K,\theta,c}^{-n}(W_{K,\theta,c})$, we define
\begin{equation} 
\label{eq:gextend}
G_{K,\theta,c}(z) = \frac{ G_{K,\theta,c} ( H^n_{K,\theta,c}(z) )}{2^n}.
\end{equation}
Let us show that this is well-defined, that is, it is independent of the specific choice of $n$. Suppose that $z\in H^{-n_{1}}_{K,\theta,c}(W_{K,\theta,c})\cap H^{-n_{2}}_{K,\theta,c}(W_{K,\theta,c})$, where $n_{1}<n_{2}$. By \eqref{eq:gktcsemiconj}, we have
\begin{align*}
G_{K,\theta,c}(z)&=\frac{G_{K,\theta,c}(H^{n_{2}}_{K,\theta,c}(z))}{2^{n_2}}\\
&=\frac{G_{K,\theta,c}(H^{n_2-n_1}_{K,\theta,c}(H^{n_{1}}_{K,\theta,c}(z)))}{2^{n_2}}\\
&=\frac{2^{n_2-n_1}G_{K,\theta,c}(H^{n_{1}}_{K,\theta,c}(z))}{2^{n_2}}\\
&=\frac{G_{K,\theta,c}(H^{n_{1}}_{K,\theta,c}(z))}{2^{n_1}}.
\end{align*}
Further, this extension of $G_{K,\theta,c}$ to $H_{K,\theta,c}^{-n}(W_{K,\theta,c})$ is continuous as all the maps in the composition given in \eqref{eq:gextend} are continuous. By Lemma \ref{lem:Big Union}, we may extend $G_{K,\theta,c}$ continuously to all of $I(H_{K,\theta,c})$ and, moreover, \eqref{eq:gextend} implies that we have \eqref{eq:gktcsemiconj} on all of $I(H_{K,\theta,c})$.

Next, we show that $G_{K,\theta,c}(z) \to 0$ as $z \to \partial I(H_{K,\theta,c})$. We may assume that the open neighborhood of infinity $W_{K,\theta,c}$ is a strictly positive distance from $\partial I(H_{K,\theta,c})$. Fix $N\in \N$ and let 
\[ U = \overline{ H^{-(N+1)}_{K,\theta,c}(W_{K,\theta,c})} \setminus H^{-N}_{K,\theta,c}(W_{K,\theta,c})  .\]
As $U$ is a compact set and $G_{K,\theta,c}$ is a continuous function, we have
\[ \sup_{z\in U} G_{K,\theta,c}(z)  = T <\infty .\]
It then follows by construction that every component of $H_{K,\theta,c}^{-1}(U)$ is contained in a bounded component of the complement of $U$. Moreover, by \eqref{eq:gextend}, for $z\in H_{K,\theta, c}^{-1}(U)$ we have
\[ G_{K,\theta,c}(z)  = \frac{ G_{K,\theta,c}(H_{K,\theta,c}) }{2} \leq \frac{T}{2}.\]
By induction, for $z\in H_{K,\theta,c}^{-n}(U)$, we have
\[ G_{K,\theta,c}(z) \leq \frac{ T}{2^n}.\]
From this we conclude that $G_{K,\theta,c}(z) \to 0$ as $z\to \partial I(H_{K,\theta,c})$. We may therefore extend $G_{K,\theta,c}$ to all of $\C$ by setting it equal to $0$ on $BO(H_{K,\theta,c})$. This completes the proof.
\end{proof}

\section{Equipotentials}
\label{sec:equi}

We use our version of Green's function to define equipotentials.

\begin{definition}
\label{def:equip}
Let $K>1$, $\theta\in (-\pi/2,\pi/2]$ and $c\in \C$. For any $t> 0$, we define
\[ E_{K,\theta,c}(t) = \{ z\in \C : G_{K,\theta,c}(z) =t \} \]
and 
\[ U_{K,\theta,c}(t) = \{ z\in \C : G_{K,\theta,c}(z) > t \}.\]
\end{definition}

We could also make these definitions for $t=0$, but in this case $E_{K,\theta,c}(0) = BO(H_{K,\theta,c})$ and $U_{K,\theta,c}(0) = I(H_{K,\theta,c})$. If the context is clear, we will just write $E(t)$ and $U(t)$. 

\begin{lemma}
\label{lem:euconj}
Let $K>1$, $\theta\in (-\pi/2,\pi/2]$ and $c\in \C$.
For any $n\in \N$ and any $t>0$, we have
\[ E(2^nt) = H_{K,\theta,c}^n(E(t)) \quad \text{ and } \quad U(2^nt) = H_{K,\theta,c}^n(U(t)).\]
\end{lemma}

\begin{proof}
Suppose that $z\in E(2^nt)$ so that $G_{K,\theta,c}(z) =2^nt$. For any $w\in \C$ such that $H^n_{K,\theta,c}(w) = z$, by Theorem \ref{thm:green}, we have
\[ t = \frac{ G_{K,\theta,c}(z) }{2^n} = \frac{G_{K,\theta,c}(H_{K,\theta,c}^n(w))}{2^n} = G_{K,\theta,c}(w)\]
and $w\in E(t)$. Thus $z\in H_{K,\theta,c}^n(E(t))$. The reverse inclusion holds analogously which gives the first result. The analogous result for $U$ follows by a judicious replacement of an equality by an inequality in the above.
\end{proof}

\begin{lemma}
\label{lem:eubdry}
Let $K>1$, $\theta\in (-\pi/2,\pi/2]$ and $c\in \C$.
Then for any $t>0$, we have $E(t) = \partial U(t)$.
\end{lemma}

\begin{proof}
By the continuity of $G_{K,\theta,c}$, it is evident that $\partial U(t) \subset E(t)$. On the other hand, if $z\in E(t) \setminus \partial U(t)$, then $z$ would be a local maximum for $G_{K,\theta,c}$. By Theorem \ref{thm:green}, $H^n_{K,\theta,c}(z)$ would also be a local maximum of $G_{K,\theta,c}$ for all $n$. However, for large $|z|$, we have $G_{K,\theta,c} = (\log \tau_{K,\theta,0}) \circ \varphi_{K,\theta,c}$. As $\varphi_{K,\theta,c}$ is quasiconformal and $\tau_{K,\theta,0}$ has no local maxima, this is a contradiction. The result follows.
\end{proof}

\begin{lemma}
\label{lem:euprops}
Let $K>1$, $\theta\in (-\pi/2,\pi/2]$ and $c\in \C$.
\begin{enumerate}[(a)]
\item There exists $T>0$ such that if $t>T$ then $E(t)$ is a simple closed curve.
\item For any $t>0$, $E(t)$ is a finite collection of closed curves and $U(t)$ is an open neighborhood of infinity. 
\item Suppose that $0<s<t$. Then $\overline{U(t)}$ is contained in $U(s)$. Moreover, every component of $E(s)$ is contained in a bounded component of $\C \setminus E(t)$.
\end{enumerate}
\end{lemma}

\begin{proof}
For part (a), suppose first that $t>0$ is large enough that $G_{K,\theta,c}^{-1}(t) = \tau_{K,\theta,c}^{-1}(e^t) \subset W_{K,\theta,c}$, recalling that this latter set is where the B\"ottcher coordinate $\varphi_{K,\theta,c}$ is defined. As $\tau_{K,\theta,c} =  \tau_{K,\theta,0}\circ \varphi_{K,\theta,c}$, we see that
\[ G_{K,\theta,c}^{-1}(t) = \varphi_{K,\theta,c}^{-1} ( \tau_{K,\theta,0}^{-1}(e^t) ).\]
As $\tau_{K,\theta,0}^{-1}(e^t)$ is a simple closed curve, and as $\varphi_{K,\theta,c}$ is quasiconformal, we see that $E(t)$ is a simple closed curve for large enough $t$.

For part (b), suppose that $t>0$ is arbitrary and suppose $n\in \N$ is chosen large enough that part(a) applies to show $E(2^nt)$ is a simple closed curve. By Lemma \ref{lem:euconj}, as $H_{K,\theta,c}^n$ is a finite degree quasiregular map, then the inverse image of $E(2^nt)$ under $H_{K,\theta,c}^n$ is a finite collection of closed curves. 

As $G_{K,\theta,c}$ is continuous, it follows that $U(t)$ is open and connected for all $t>0$. Moreover, by Lemma \ref{lem:eubdry}, $U(t)$ is the unbounded component of the complement of $E(t)$ and hence is a neighborhood of infinity.

For part (c), if $z\in U(t)$, then $G_{K,\theta,c}(z)>t>s$, so $z\in U(s)$. Suppose that $z \in \partial U(t)$. Then $G_{K,\theta,c}(z)\leq t$. Any open neighborhood of $z$ contains an element $w$ such that $G_{K,\theta,c}(w)>t$, so $G_{K,\theta,c}(z)$ must equal $t$ since $G_{K,\theta,c}$ is continuous. It follows that $\partial U(t) \subset U(s)$. We conclude that $\overline{ U(t)} \subset U(s)$. Moreover, by Lemma \ref{lem:eubdry}, we see that $E(t) \subset U(s)$ and the final claim follows.
\end{proof}

Observe that each component of $E(t)$ need not be a simple closed curve. For example, if $c\notin \mathcal{M}_{K,\theta}$ then there exists $t_0>0$ such that $E(t_0)$ contains the critical point $0$ of $H_{K,\theta,c}$. Then $E(t_0)$ will be a topological figure eight. The image of $E(t_0)$ under $H_{K,\theta,c}$ will be a simple closed curve. We make these observations more precise in the following results.

\begin{lemma}
\label{lem:mandel}
Let $K>1$ and $\theta \in (-\pi/2,\pi/2]$. If $c\in \mathcal{M}_{K,\theta}$, then for any $t>0$, $U_{K,\theta,c}(t) \cup \{ \infty \}$ is a simply connected subdomain of $\C_{\infty}$ and $E_{K,\theta,c}(t)$ has one component.
\end{lemma}

\begin{proof}
Fix $c\in \mathcal{M}_{K,\theta}$. For $t>0$, set $U(t) = U_{K,\theta,c}(t)$.
By Lemma \ref{lem:euprops} (a) and (b), there exists $T>0$ such that if $t>T$, then $U(t)$ has the claimed properties.

Suppose now that $t>0$ is arbitrary and find $n\in \N$ such that $2^nt > T$. We will apply the Riemann-Hurwitz formula, Theorem \ref{thm:rh}, with $f=H_{K,\theta,c}^n$, $D_1= U(t)\cup\{\infty \}$ and $D_2 = U(2^nt)\cup \{ \infty \}$. We observe that $H_{K,\theta,c}^n$ is a proper map, as if $X\subset D_2$ is compact, then $X \subset U(s)$ for some $s>2^nt$ and so $H^{-n}_{K,\theta,c}(X) \subset U(2^{-n}s)$ which is compactly contained in $D_1$.

The degree of $H_{K,\theta,c}^n$ is $2^n$. As $c\in \mathcal{M}_{K,\theta}$, $0\notin I(H_{K,\theta,c})$ and so the number of branch points of $H^n_{K,\theta,c}$ in $D_1$ counting multiplicity just comes from the point at infinity, and is thus $L = 2^n-1$. As $D_2$ has only one boundary component, we have $d_2 = 1$. Solving
\[ 2-d_1 = 2^n(2-1) - 2^n+1\]
for $d_1$ yields $d_1 = 1$. We conclude that $U(t)$ has one boundary component and hence $U(t) \cup \{\infty \}$ is simply connected.
\end{proof}

\begin{lemma}
\label{lem:notmandel}
Let $K>1$ and $\theta \in (-\pi/2,\pi/2]$. If $c\notin \mathcal{M}_{K,\theta}$, set $t_0 = G_{K,\theta,c}(0) > 0$. If $t\geq t_0$ then $U_{K,\theta,c}(t) \cup \{ \infty \}$ is a simply connected subdomain of $\C_{\infty}$ and $E_{K,\theta,c}(t)$ has one component. If $t<t_0$, let 
\begin{equation}
\label{eq:mthing} 
m = \left \lceil \frac{ \ln(t_0/t) }{\ln 2}\right \rceil \in \N.
\end{equation}
Then $E_{K,\theta,c}(t)$ has $2^m$ components and $U_{K,\theta,c}(t)$ is $2^m$-connected.
\end{lemma}

\begin{proof}
Fix $c\notin \mathcal{M}_{K,\theta}$ and for $t>0$ set $U(t) = U_{K,\theta,c}(t)$.
The proof of Lemma \ref{lem:mandel} shows that as long as the critical point $0$ does not lie in $U(t)$, then $U(t) \cup \{ \infty \}$ is simply connected. With $t_0 = G_{K,\theta,c}(0)$, $E(t_0)$ is a figure-eight shape, but $U(t_0) \cup \{\infty\}$ is still simply connected.

Now suppose that $t<t_0$ and $m$ is the smallest integer such that $2^mt \geq t_0$. Then $m$ is given by \eqref{eq:mthing}. We will apply the Riemann-Hurwitz formula to the proper map $H^m_{K,\theta,c}$ with $D_1 = U(t) \cup \{\infty \}$ and $D_2 =U(2^mt) \cup \{ \infty \}$. The critical points of $H^m_{K,\theta,c}$ contained in $D_1$ are at infinity, with multiplicity $2^m-1$, and in the set $\{ H_{K,\theta,c}^{-j} (0) : j=0 ,\ldots, m-1 \}$, each with multiplicity $1$. As $H_{K,\theta,c}^{-j} (0)$ contains $2^j$ elements, we have 
\[L = (2^m-1) + (1 + \ldots + 2^{m-1}) = 2(2^m-1).\]
The number of boundary components of $D_2$ is $d_2 = 1$ and so the Riemann-Hurwitz formula gives
\[ 2-d_1 = 2^m ( 2-1) - 2(2^m-1),\]
which implies $d_1 = 2^m$ as required.
\end{proof}

We are now in a position to prove that the bounded orbit set contains either one, or uncountably many, components.

\begin{proof}[Proof of Theorem \ref{thm:conn}]
By Lemma \ref{lem:euprops} (c), we can reformulate the bounded orbit set as
\begin{equation} 
\label{eq:bointersection}
BO(H_{K,\theta,c}) = \bigcap _{t>0} \left ( \C \setminus U_{K,\theta,c}(t) \right ).
\end{equation}
If $c\in \mathcal{M}_{K,\theta}$, then Lemma \ref{lem:mandel} implies that $BO(H_{K,\theta,c})$ contains one component.

On the other hand, suppose that $c\notin \mathcal{M}_{K,\theta}$. Set $t_0 = G_{K,\theta,c}(0) >0$ and choose $t_1 \in (t_0 , 2t_0)$. Then by the proof of Lemma \ref{lem:notmandel}, $E(t_1)$ is a simple closed curve and $E(t_1/2)$ consists of two simple closed curves. Denote by $X_0$ the bounded component of the complement of $E(t_1)$ and by $Y_{j}$, for $j=1,2$, the two bounded components of the complement of $E(t_1/2)$. Then $H_{K,\theta,c}$ restricted to each $Y_{j}$ is a homeomorphism onto $X_0$ and we may therefore define two branches of the inverse, say
\[ g_j : X_0 \to Y_{j}\]
for $j=1,2$. Evidently, $\overline{Y_{j}} \subset X_0$ for $j=1,2$. 

Here it is convenient to introduce symbolic notation. Given an integer $k\geq 0$, we denote by $\{g_1,g_2\}^k$ the set of words formed from the alphabet $\{g_1,g_2\}$ that have length exactly $k$. Conventionally, we set $\{g_1,g_2\}^0 = \{\varepsilon\}$ where $\varepsilon$ is the empty word. We also denote by $\{g_1,g_2\}^* = \bigcup_{k\geq 0}\{g_1,g_2\}^k$ the set of all finite words formed from $\{g_1,g_2\}$. Given a word $w \in \{g_1,g_2\}^*$, we denote by $|w|$ the length of $w$ with the convention $|\varepsilon| = 0$.

With this convention, for $w\in \{g_1,g_2 \}^*$, we can define 
\[ X_w = w(X_0),\]
so that, for example, $X_{g_1g_2} = g_1(g_2(X_0))$. By Lemma \ref{lem:euprops} (c), we can rewrite \eqref{eq:bointersection} as
\[ BO(H_{K,\theta,c}) = \bigcap _{w \in \{g_1,g_2\}^*} X_w.\]
Every component of $BO(H_{K,\theta,c})$ can be identified with an infinite word from the alphabet $\{g_1,g_2\}$, and by a standard diagonal argument, the number of such components is uncountable.
\end{proof}

\section{Fixed and periodic points}
\label{sec:fixed}

In this section, we will focus on the case where $c\notin\mathcal{M}_{K,\theta}$ and so $BO(H_{K,\theta,c})$ has uncountably many components.

\begin{proof}[Proof of Theorem \ref{thm:perpts}]
Fix $K>1$, $\theta \in (-\pi/2, \pi/2]$, $c\notin \mathcal{M}_{K,\theta}$ and $n\in \N$ such that $H_{K,\theta,c}$ has at least $2^n+1$ periodic points of period $n$. Then $H^n_{K,\theta,c}$ has at least $2^n+1$ fixed points. Using the symbolic notation established in the proof of Theorem \ref{thm:conn}, and Lemma \ref{lem:notmandel}, the components of $H^{-n}_{K,\theta,c}(X_0)$ can be enumerated by $w\in \{g_1,g_2\}^n$. There are therefore $2^n$ components.

By the pigeonhole principle, at least two of the fixed points of $H_{K,\theta,c}^n$, say $z_1\neq z_2$, must be in the same component, say $A = w(X_0)$. 
Consider now $w(A)$. As $w$ represents a homeomorphism, $w(A)$ is connected and contains $z_1$ and $z_2$. Repeating this argument, we see that 
\[ \{ z_1,z_2 \} \subset \bigcap_{k\geq 1} w^k(X_0),\]
which means that $z_1,z_2$ are contained in the same component of $BO(H_{K,\theta,c})$. We conclude that the bounded orbit set is not totally disconnected.
\end{proof}

We will now show that it is possible for the hypothesis of Theorem \ref{thm:perpts} to hold. Here, we will set $\theta = 0$ and $c\in \R$ so that the mapping we consider is
\[ H_{K,0,c}(x+iy) = (K^2x^2-y^2+c ) +i(2Kxy).\]
Solving for fixed points, by equating real and imaginary parts we obtain the two equations
\[ K^2x^2-y^2+c = x, \quad \text{ and } 2Kxy= y.\]
The second of these clearly has solutions $y=0$ or $x=(2K)^{-1}$. In the first case when $y=0$, the first equation yields
\[ K^2x^2-x+c=0 \implies x = \frac{ 1 \pm \sqrt{1-4K^2c} }{2K^2}.\]
This has real solutions if and only if $c\leq (4K^2)^{-1}$. On the other hand, if $x=(2K)^{-1}$, then the first equation yields
\[ \frac{1}{4} - y^2 - \frac{1}{2K} + c =0 \implies y = \pm \sqrt{ \frac{1}{4} - \frac{1}{2K} + c }.\]
This has real solutions if and only if $c\geq (2K)^{-1} - 1/4$. It is worth pointing out that we always have
\[ \frac{1}{2K} - \frac{1}{4} \leq \frac{1}{4K^2}\]
as rearranging this yields $(K-1)^2 \geq 0$.

By analyzing the four cases that can occur here, we see that $H_{K,0,c}$ has four fixed points if $c$ lies in the interval 
\[ \left( \frac{1}{2K} - \frac{1}{4} , \frac{1}{4K^2}\right ),\]
three fixed points if $c$ is at either endpoint of this interval, and two fixed points otherwise.

Next, we show that there are parameters $c$ for which $H_{K,0,c}$ has four fixed points and for which $c\notin \mathcal{M}_{K,0}$. To this end, recall from \cite[Theorem 6.4]{FG10} that 
\[ \mathcal{M}_{K,0} \cap \R = \left [ -\frac{2}{K^2} , \frac{1}{4K^2} \right ].\]
The right hand endpoint of this interval agrees with the endpoint of the interval above, and so we need to analyze the left hand endpoints. We want
\[ \frac{1}{2K} - \frac{1}{4} < -\frac{2}{K^2},\]
which after some elementary algebra simplifies to $K>4$, recalling that we assume $K>1$. We conclude that if this is so, then for 
\[ c \in \left ( \frac{1}{2K} - \frac{1}{4} , -\frac{2}{K^2} \right )\]
the map $H_{K,0,c}$ has four fixed points, and hence by Theorem \ref{thm:perpts}, $BO(H_{K,0,c})$ has uncountably many components, but is not totally disconnected.

\section{Classification of fixed points}
\label{sec:classify}

It is clear that if $z_0$ is a fixed point of $P_c(z) = z^2+c$, then it is attracting, neutral or repelling according to whether $|z_0|$ is less than, equal to, or greater than $\tfrac12$ respectively. Here we establish the analogous result for $H_{K,0,c}$. We focus on the case where $\theta = 0$ for the sake of clarity of exposition, but the generalization to arbitrary $\theta$ may be computed similarly. Now we fix $K>1$ and $c\in \C$, although the role of $c$ in what follows is unimportant, and set $H = H_{K,0,c}$.

Recall that 
\[ H(x+iy)=(Kx+iy)^2+c=K^2x^2-y^2+\Re(c)+[2Kxy+\Im(c)]i\]
and so the derivative of $H$ is \lefteqn{\begin{bmatrix}2K^2x & -2y\\2Ky & 2Kx \end{bmatrix}.}

The eigenvalues of this matrix are solutions of the quadratic 
\[ \lambda^2+(-2K^2x-2Kx)\lambda+4K^3x^2+4Ky^2=0,\]
which are
\begin{equation} 
\label{eq:evals}
\lambda=K^2x+Kx \pm [(K^2x+Kx)^2-4(K^3x^2+Ky^2)]^{1/2} .
\end{equation}
We let $\lambda_1$ and $\lambda_2$ denote the two eigenvalues with the convention that $|\lambda_1|\leq |\lambda_2|$. We aim to partition the regions of $\mathbb{C}$ according to the absolute values of $|\lambda_1|$ and $|\lambda_2|$ in order to classify the type of fixed point at a given point $(x_0,y_0)$.
We point out that if the eigenvalues at $x+iy$ are given by $\lambda_1,\lambda_2$, and if $s\in \R$, then the eigenvalues at $(sx) +i (sy)$ are given by
\begin{equation}
\label{eq:scaleevals} 
s\lambda_1, s\lambda_2.
\end{equation}

First, we find where the eigenvalues are real or non-real.

\begin{lemma}
\label{lem:fp1}
Let $S_K^+$ be the sector given by 
\[ y > \frac{ \sqrt{K}(K-1)x }{2} \text{ and } y > - \frac{ \sqrt{K}(K-1)x }{2}, \]
and let $S_K^-$ be the sector given by
\[ y < \frac{ \sqrt{K}(K-1)x }{2} \text{ and } y < - \frac{ \sqrt{K}(K-1)x }{2}. \]
If $x+iy$ lies in either $S_K^{\pm}$, then $\lambda_1$ and $\lambda_2$ are non-real complex conjugates. Otherwise $\lambda_1$ and $\lambda_2$ are real. Moreover, if $x+iy$ lies on one of the two lines $L_K^{\pm}$ given by
\[ y = \pm \frac{ \sqrt{K}(K-1)x }{2},\]
then the two eigenvalues are the same (and real).
\end{lemma}

\begin{proof}
It follows from \eqref{eq:evals} that $\lambda_1$ and $\lambda_2$ are non-real complex conjugates if and only if 
\[ (K^2x+Kx)^2-4(K^3x^2+Ky^2) < 0.\]
We can compute that
\begin{align*}
(K^2x+Kx)^2-4(K^3x^2+Ky^2) &= K^4x^2-2K^3x^2+K^2x^2-4Ky^2 \\
&= K^2x^2(K-1)^2-4Ky^2 \\
&= (K(K-1)x-2\sqrt{K}y)(K(K-1)x+2\sqrt{K}y).
\end{align*}
The first part of the lemma follows by assigning one of these factors to be strictly positive and one to be strictly negative. The second part follows immediately from \eqref{eq:evals}.
\end{proof}

In particular, it follows from Lemma \ref{lem:fp1} that on the imaginary axis, away from $0$, the eigenvalues are always non-real. When the eigenvalues are complex conjugates, we can compute their absolute values as follows.

\begin{lemma}
\label{lem:fp2}
Suppose that $\lambda_1$ and $\lambda_2$ are complex conjugates. Then $|\lambda_1| = |\lambda_2| = 1$ when $x+iy$ lies on the curve $\gamma_K$ given by
\[ 4K^3x^2 + 4Ky^2=1.\]
\end{lemma}

Clearly $\gamma_K$ is the boundary curve of an ellipse.
Note that via \eqref{eq:scaleevals}, it follows that $|\lambda_1| = |\lambda_2| = r$ when
\[ 4K^3x^2 + 4Ky^2=r^2.\]

\begin{proof}
As $|\lambda_1| = |\lambda_2| = 1$ and $\lambda_1$ is a complex conjugate of $\lambda_2$, it follows that $\lambda_1\lambda_2 =1$. Therefore, we have
\[ (K^2x+Kx)^2-[(K^2x+Kx)^2-4(K^3x^2+Ky^2)]=1.\]
Simplifying this expression completes the proof.
\end{proof}

We now turn to the case where the eigenvalues are real.

\begin{lemma}
\label{lem:fp3}
The set where one of the eigenvalues has values $\pm 1$ is given by the pair of ellipses $E_K^{\pm}$ with centers 
\[ w_{\pm} = \pm  \frac{ K+1}{4K^2} ,\]
major and minor semi-axes oriented vertically and horizontally respectively, with lengths
\[ v = \frac{K-1}{4K} \text{ and } h = \frac{K-1}{4K^2}.\]
\end{lemma}

\begin{proof}
Suppose $\lambda$ is either $\pm 1$. Then it follows from \eqref{eq:evals} that
\[ (\lambda-(K^2x+Kx))^2 = (K^2x+Kx)^2-4(K^3x^2+Ky^2).\]
By expanding out the brackets, we obtain
\[ (K^2x+Kx)^2-2\lambda(K^2x+Kx)+\lambda^2=(K^2x+Kx)^2-4(K^3x^2+Ky^2).\]
Collecting like terms together and simplifying, this yields
\[ 4K^3x^2+(-2K^2\lambda-2K\lambda)x + 4Ky^2+\lambda^2 =0 .\]
Completing the square in the $x$-variable, we get
\begin{equation} 
\label{eq:fp2-1}
4K^3\left (x-\left(\frac{\lambda(K+1)}{4K^2}\right) \right )^2+4Ky^2=4K^3\left(\frac{\lambda(K+1)}{4K^2}\right)^2-\lambda^2 .
\end{equation}
Now, the right hand side of \eqref{eq:fp2-1} is 
\begin{align*}
4K^3\left(\frac{\lambda(K+1)}{4K^2}\right)^2-\lambda^2&= 4K^3\lambda^2\left(\frac{1}{16K^2}+\frac{1}{8K^3}+\frac{1}{16K^4}\right )-\lambda^2\\
&= \frac{K\lambda^2}{4}+\frac{\lambda^2}{2}+\frac{\lambda^2}{4K}-\lambda^2 \\
&=\frac{\lambda^2}{4K}(K-1)^2.
\end{align*}
As $\lambda = \pm 1$ and $K>1$, this term is strictly positive. From \eqref{eq:fp2-1}, we therefore obtain
\begin{equation}
\label{eq:ellipses} 
\frac{\left(x-\left(\frac{\lambda(K+1)}{4K^2}\right)\right)^2}{\left(\frac{\lambda(K-1)}{4K^2}\right)^2}+\frac{y^2}{\left(\frac{\lambda(K-1)}{4K}\right)^2}=1.\end{equation}
This gives the equations of the required ellipses.
\end{proof}

As $w_+ - h = (2K^2)^{-1}>0$, it follows that $E_K^+$ is contained in the right half-plane, and $E_K^-$ in the left half-plane.

\begin{lemma}
\label{lem:fp4}
The intersection of one the ellipses $E_K^{\pm}$ and one of the lines $L_K^{\pm}$ is a unique point given by
\[ \left ( \pm \frac{1}{K(K+1)} , \pm \frac{ K-1}{2\sqrt{K}(K+1) } \right ),\]
for an appropriate choice of $\pm$ in both coordinates to correspond to the four combinations of choices of lines and ellipses. Moreover, the intersection of $\gamma_K$ with $L_K^{\pm}$ occurs at precisely the same four points.
\end{lemma}

\begin{proof}
We will work with $E_K^+$ and $L_K^+$. The other cases follow analogously. Plugging the equation $y = \sqrt{K}(K-1)x/2$ for $L_K^+$ into the equation for the ellipse $E_K^+$ from \eqref{eq:ellipses}, we obtain
\[ K^2 \left ( x - \left ( \frac{K+1}{4K^2} \right ) \right )^2 + \frac{ K(K-1)^2x^2}{4} = \frac{(K-1)^2}{16K^2}.\]
Multiplying out the brackets, collecting like terms and simplifying, we obtain
\[ x^2 \left ( \frac{ K(K+1)^2}{4} \right ) - x \left ( \frac{K+1}{2} \right ) + \frac{1}{4K} = 0.\]
This factorizes as
\[ \left ( \frac{ x\sqrt{K} (K+1) }{2} - \frac{1}{2\sqrt{K} } \right )^2 = 0,\]
which has the unique solution $x = 1/(K(K+1))$. From this we obtain the value for $y$ given in the statement of the lemma.

For the second part of the lemma, computing the intersection of $\gamma_K$ with $L_K^{\pm}$ gives
\[ 4K^3x^2 + K^2(K-1)^2x^2=1.\]
This simplifies to 
\[ x^2K^2(K+1)^2 = 1,\]
from which the second part of the lemma follows.
\end{proof}

\begin{figure}[h]
\begin{center}
\includegraphics[width=3in]{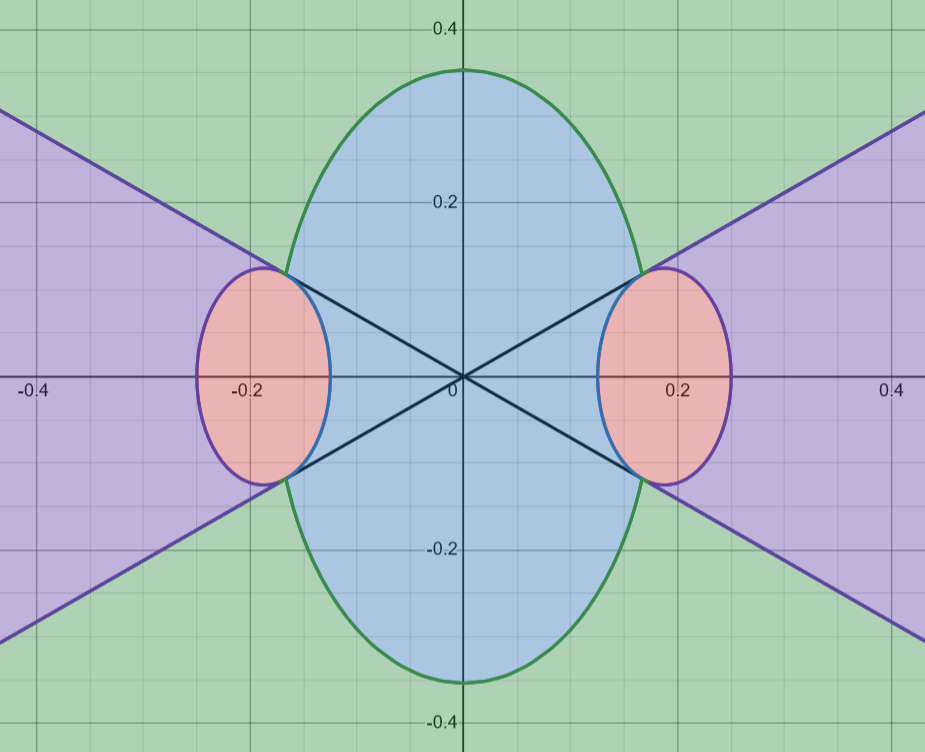}
\end{center}
\caption{The classification for $K=2$. The red ellipses are $E_2^{\pm}$ where fixed points are saddle points. The blue region is $A_2 \setminus \overline{( E_2^+ \cup E_2^-)}$ where fixed points are attracting. The green region gives fixed points that are repelling with complex conjugate eigenvalues. The purple region gives fixed points that are repelling with real eigenvalues.}
\label{fig:cicada}
\end{figure}

We can now put everything together to classify the fixed points of $H$ in terms of their location. See Figure \ref{fig:cicada} for the example $K=2$.

\begin{theorem}
\label{thm:classify}
Let $K>1$, let $c\in \C$ and let $H = H_{K,0,c}$. Denote by $A_K$ the bounded component of the complement of the ellipse given by $\gamma$ in Lemma \ref{lem:fp2}. Suppose that $z = x+iy$ is a fixed point of $H$. Then:
\begin{enumerate}[(a)]
\item $z$ is attracting if $z$ lies in $A_K \setminus \overline{( E_K^+ \cup E_K^-)}$;
\item $z$ is repelling if $z$ lies in $\C \setminus \overline{ (A_K \cup E_K^+ \cup E_K^-)}$;
\item $z$ is a saddle point if $z$ lies in $E_K^{\pm} $.
\end{enumerate}
\end{theorem}

\begin{proof}
By Lemma \ref{lem:fp1}, $\lambda_1,\lambda_2$ are non-real complex conjugates in the sectors $S_K^{\pm}$. Then by Lemma \ref{lem:fp2} and \eqref{eq:scaleevals}, cases (a) and (b) in $S_K^{\pm}$ follow directly. Denote by $\Sigma_K^{\pm}$ the two sectors given by the complement of $S_K^{\pm}$. By Lemma \ref{lem:fp3}, case (c) in $\Sigma_K^{\pm}$ follows directly. By Lemma \ref{lem:fp4}, as $\gamma_K$, $L_K^{\pm}$ and $\partial E_K^{\pm}$ all meet in common points, it follows that $\overline{E_k^{\pm}}$ partition each of $\Sigma_K^{\pm}$ into two components, one where both eigenvalues are less than one, and one where both eigenvalues are greater than one. These cases complete (a) and (b).
\end{proof}

\section{Examples}
\label{sec:examples}

In this section, we will construct the examples that give Theorem \ref{thm:examples}.
We already showed in Section 5 that it is possible for $H_{K,0,c}$ to have two, three or four fixed points. In general, solving $H_{K,\theta,c}(z) = z$ yields two quadratic equations in $x$ and $y$. Bezout's Theorem then implies that the maximum number of solutions is four unless there is a curve of fixed points of $H_{K,\theta,c}$. However, this latter case would imply a curve of points where the derivative of $H_{K,\theta,c}$ is the identity, which does not occur. This proves Theorem \ref{thm:examples} (a).

For part (b), consider the parameters $K = 0.5$ $\theta =0$ and $c = -3/2-i/2$. We typically choose $K>1$, but our choice of parameters is conjugate to $K=2$, $\theta  =\pi/2$ and $c = -3/8-i/8$ by \cite[Proposition 3.1]{FG10}. We can compute that a fixed point $z_0$ of $H_{1/2,0,-3/2-i/2}$ lies at approximately $-1.195 - 0.228i$, see Figure \ref{fig:3}.

Computing the eigenvalues of $H_{1/2,0,-3/2-i/2} ' (z_0)$ yields two complex conjugate eigenvalues $\lambda_{\pm}$ which are approximately $-0.896 \pm 0.121i$. This yields
\[ | \lambda _{\pm} | \approx 0.904 \]
and thus we conclude that $z_0$ is an attracting fixed point of $H_{1/2,0,-3/2-i/2}$. However, we may computationally check that
\begin{equation}
\label{eq:examples1} 
H_{1/2,0,-3/2-i/2}^8(z_0) \approx -8.875 - 1.193i.
\end{equation}
By \cite[Theorem 6.3]{FG10}, the Mandelbrot set $\mathcal{M}_{1/2,0}$ can be characterized as the set of $c\in \C$ for which $H_{1/2,0,c}^n(0) \leq 8$ for all $n\in \N$. Combining this with \eqref{eq:examples1}, we conclude that $-3/2-i/2 \notin \mathcal{M}_{1/2,0}$, see Figure \ref{fig:1}.

\begin{figure}[h]
\begin{center}
\includegraphics[width=5in]{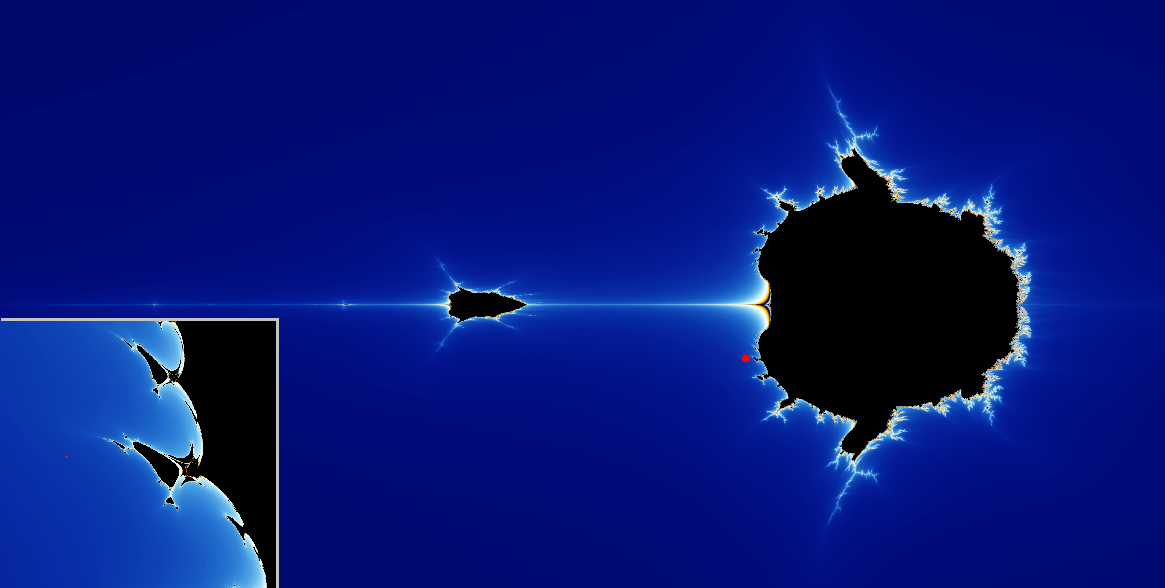}
\end{center}
\caption{The Mandelbrot set $\mathcal{M}_{1/2,0}$ with $c = -3/2-i/2$ marked and a zoom inset.}
\label{fig:1}
\end{figure}

This example completes the proof of Theorem \ref{thm:examples} (b). For part (c), we already know from Theorem \ref{thm:classify} that $H_{K,0,c}$ may have saddle fixed points. We will look more in depth at the example corresponding to $K = 5$, $\theta = 0$ and $c = -0.1$. See Figure \ref{fig:6}.

\begin{figure}[h]
\begin{center}
\includegraphics[width=1in]{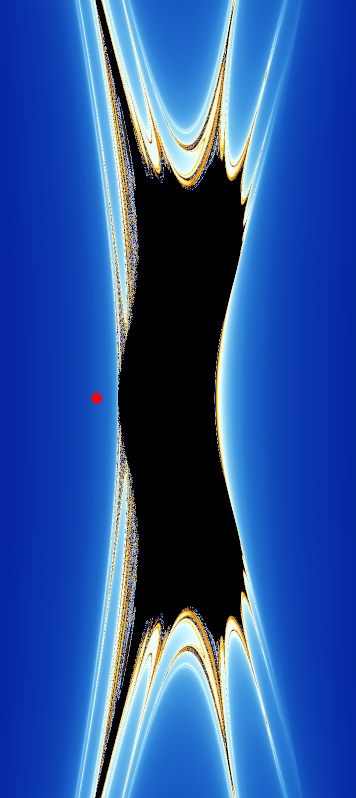}
\end{center}
\caption{Part of the Mandelbrot set $\mathcal{M}_{5,0}$ with $c=-0.1$ marked.}
\label{fig:6}
\end{figure}

We have
\[ H_{5,0,-0.1}(z) = 25x^2 - y^2 - 0.1 + 10ixy,\]
from which we see that restricting to the real axis gives the real quadratic polynomial $H_{5,0,-0.1}(x) = 25x^2-0.1$. In particular, on the real axis, $H_{5,0,-0.1}$ is conjugate to the quadratic polynomial $P_{-2.5}(z) = z^2-2.5$. As $-2.5$ is not in the Mandelbrot set, it follows that the Julia set of $P_{-2.5}$ is a Cantor subset of $\R$. We deduce that $BO(H_{5,0,-0.1}) \cap \R$ is also a Cantor subset of $\R$.

Now, we can compute that $H_{5,0,-0.1}$ has a fixed point at $z_0 \approx 0.086$. Computing the eigenvalues of $H'_{5,0,-0.1}(z_0)$, we obtain $\lambda_1,\lambda_2$ which are approximately $4.317$ and $0.863$ respectively. Evidently this implies that $z_0$ is a saddle fixed point of $H_{5,0,-0.1}$. The eigenvectors for $\lambda_1,\lambda_2$ are $(1,0)$ and $(0,1)$ respectively, which means that $H_{5,0,-0.1}$ is repelling in the $x$-direction and attracting in the $y$-direction near $z_0$.

More precisely, the version of the Hartman-Grobman Theorem from \cite{Har60} provides a $C^1$ conjugacy in a neighborhood $\mathcal{N}$ of $z_0$ of the diffeomorphism $H_{5,0,-0.1}|_{\mathcal{N}}$ to the derivative $A:= H'_{5,0,-0.1}(z_0)$. That is, there exists a $C^1$ invertible function $\varphi$ such that
$\varphi(z_0) = 0$ and
\begin{equation} 
\label{eq:conj}
\varphi \circ H_{5,0,-0.1} \circ \varphi^{-1} = A.
\end{equation}
Then the stable manifold $W^S$ is given by $\varphi^{-1}( \varphi (\mathcal{N}) \cap i\R )$ and the unstable manifold $W^U$ is given by  $\varphi^{-1}( \varphi (\mathcal{N}) \cap \R )$.

\begin{proposition}
\label{prop:stable manifold}
The component of $BO(H_{5,0,-0.1}) \cap \mathcal{N}$ containing $z_0$ is precisely the stable manifold $W^S$ and is, in particular, a curve.
\end{proposition}

\begin{proof}
As $BO(H_{5,0,-0.1})\cap \mathbb{R}$ is totally disconnected, for any $\delta>0$, the real interval $(z_{0}-\delta,z_{0}+\delta)$ intersects $I(H_{5,0,-0.1})$. Choose such a point $z_{\delta}$. Then as $I(H_{5,0,-0.1})$ is open and connected, it is path connected. In particular, there is a closed smooth path $\gamma_{\delta}$ in $\mathbb{C}_{\infty}$ that connects $z_{\delta}$ to $\infty$ and otherwise is contained in $I(H_{5,0,-0.1})$. Choose $\delta>0$ small enough that $(z_{0}-\delta,z_{0}+\delta)\subset \mathcal{N}$. Let $T=\varphi(\gamma_{\delta}\cap \mathcal{N})$. By construction, for $z\in \varphi(N)$ with non-zero imaginary part we have 
	
	$$|\Im(A^{-m}(z))|=\frac{|\Im(z)|}{\lambda_{2}^{m}} \to \infty$$ 
as $m\to \infty$ because $\lambda_{2}<1$. Moreover, for any $z\in \varphi(N)$,
	
	$$\Re(A^{-m}(z))=\frac{\Re(z)}{\lambda_{1}^{m}} \to 0$$
as $m\to \infty$ because $\lambda_1>1$.	
	
	Since $\varphi(\mathcal{N})$ is open, we may assume that $T$ contains elements with positive imaginary parts and elements with negative imaginary parts. Therefore, for any $\varepsilon>0$, there exists $M\in \mathbb{N}$ such that for $m\geq M$ we have $\operatorname{dist}(A^{-m}(T)\cap \varphi(\mathcal{N}),i\mathbb{R})<\varepsilon$ and $A^{-m}(T)$ and the boundary of $\varphi(\mathcal{N})$ share an element with positive imaginary part in common and, moreover, an element with negative imaginary part in common. In particular, $A^{-m}(T)\cap \varphi(\mathcal{N})$ separates $\varphi(\mathcal{N})$ into two components. Evidently, as $m$ increases, $A^{-m}(T)\cap \varphi(\mathcal{N})$ accumulates on $i\mathbb{R}\cap \varphi(\mathcal{N})$.
	
	Transferring this back to the dynamical plane of $H_{5,0,-0.1}$, using the conjugacy \eqref{eq:conj}, and the complete invariance of $I(H_{5,0,-0.1})$, we obtain curves $\varphi^{-1}(A^{-m}(T)\cap \varphi(\mathcal{N}))$ contained in $\mathcal{N}\cap I(H_{5,0,-0.1})$ which accumulate on $W^{S}\cap \mathcal{N}$. As this argument holds for both cases $z_{\delta}<z_{0}$ and $z_{\delta}>z_{0}$, we accumulate such curves on $W^{S}$ from both left and right. As $W^{S}\subset BO(H_{5,0,-0.1})$, this completes the proof.  
\end{proof}

We can extend the conjugacy $\varphi$ along the stable manifold via \eqref{eq:conj} all the way to a pair of repelling fixed points at $w_{\pm} = 1/10 \pm i / 2\sqrt{5}$.
This implies that the component of $BO(H_{5,0,-0.1})$ containing $z_0$ is a curve connecting $w_+$ to $w_-$ passing through $z_0$. However, we do not know if this is the whole component.

The same argument in Proposition \ref{prop:stable manifold} can be used at periodic points of $H_{5,0,-0.1}$ contained in the real axis. Moreover, we can use complete invariance to take pre-images of $W^S$. In both cases, we see that the intersection of components of $BO(H_{5,0,-0.1})$ with a neighborhood of the real axis yields curves. See Figure \ref{fig:4} and a zoom back in Figure \ref{fig:5}.

\begin{figure}[h]
\begin{center}
\includegraphics[width=1in]{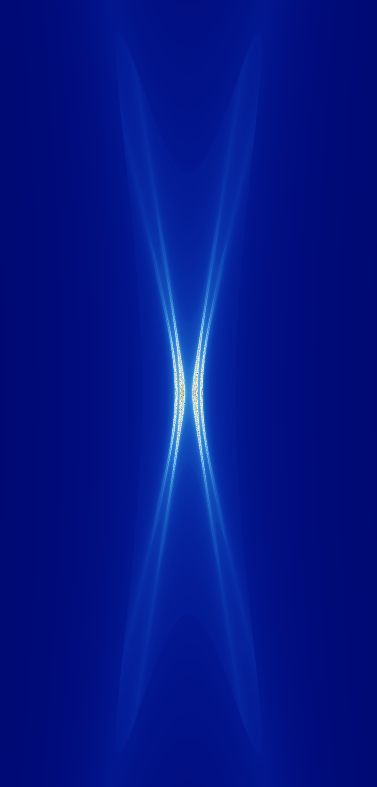}
\end{center}
\caption{The dynamical picture for $H_{5,0,-0.1}$.}
\label{fig:4}
\end{figure}

{\bf Acknowledgements:} The authors thanks James Waterman for many interesting conversations on the topic of this paper. Figure \ref{fig:cicada} was created using Desmos. The other figures were created using Ultra Fractal 6.

\end{document}